\documentclass{amsart}

\usepackage{amsthm, amsfonts, amssymb, amsmath, latexsym, enumerate, times, graphicx, amsxtra,amscd} 
\usepackage{xypic}
\usepackage[mathscr]{eucal}
\usepackage{xcolor}
\usepackage{a4}
\usepackage[normalem]{ulem}
\usepackage[latin1]{inputenc}
\usepackage{mathrsfs}
\usepackage{mathtools}
\usepackage{letltxmacro}
\usepackage{enumerate}
\usepackage{array}

\usepackage{comment}
\usepackage{todonotes}
\usepackage{tikz-cd}
\usetikzlibrary{arrows.meta,positioning}

\LetLtxMacro\orgvdots\vdots
\LetLtxMacro\orgddots\ddots

\newcounter{rowcntr}[table]
\renewcommand{\therowcntr}{\roman{rowcntr}}

\newcolumntype{N}{>{\refstepcounter{rowcntr}(\therowcntr)}c}

\makeatletter
\DeclareRobustCommand\vdots{%
	\mathpalette\@vdots{}%
}
\newcommand*{\@vdots}[2]{%
	\sbox0{$#1\cdotp\cdotp\cdotp\m@th$}%
	\sbox2{$#1.\m@th$}%
	\vbox{%
		\dimen@=\wd0 %
		\advance\dimen@ -3\ht2 %
		\kern.5\dimen@
		\dimen@=\wd2 %
		\advance\dimen@ -\ht2 %
		\dimen2=\wd0 %
		\advance\dimen2 -\dimen@
		\vbox to \dimen2{%
			\offinterlineskip
			\copy2 \vfill\copy2 \vfill\copy2 %
		}%
	}%
}
\DeclareRobustCommand\ddots{%
	\mathinner{%
		\mathpalette\@ddots{}%
		\mkern\thinmuskip
	}%
}
\newcommand*{\@ddots}[2]{%
	\sbox0{$#1\cdotp\cdotp\cdotp\m@th$}%
	\sbox2{$#1.\m@th$}%
	\vbox{%
		\dimen@=\wd0 %
		\advance\dimen@ -3\ht2 %
		\kern.5\dimen@
		\dimen@=\wd2 %
		\advance\dimen@ -\ht2 %
		\dimen2=\wd0 %
		\advance\dimen2 -\dimen@
		\vbox to \dimen2{%
			\offinterlineskip
			\hbox{$#1\mathpunct{.}\m@th$}%
			\vfill
			\hbox{$#1\mathpunct{\kern\wd2}\mathpunct{.}\m@th$}%
			\vfill
			\hbox{$#1\mathpunct{\kern\wd2}\mathpunct{\kern\wd2}\mathpunct{.}\m@th$}%
		}%
	}%
}
\makeatother

\usepackage{array}
\usepackage{comment}
\usepackage{paralist}
\usepackage{xcolor}

\newtheorem{theorem}{Theorem}[section]
\newtheorem{lemma}[theorem]{Lemma}

\newtheorem{corollary}[theorem]{Corollary}
\newtheorem{proposition}[theorem]{Proposition}

\theoremstyle{definition}

\newtheorem{remark}[theorem]{Remark}

\newcommand{\calB}{{\mathcal B}}

\newcommand{\calH}{{\mathcal H}}
\newcommand{\calI}{{\mathcal I}}
\newcommand{\calL}{{\mathcal L}}
\newcommand{\calM}{{\mathcal M}}
\newcommand{\calP}{{\mathcal P}}
\newcommand{\calO}{{\mathcal O}}
\newcommand{\calS}{{\mathcal S}}
\newcommand{\calT}{{\mathcal T}}

\newcommand{\bbC}{{\mathbb C}}
\newcommand{\bbD}{{\mathbb D}}

\newcommand{\bbP}{{\mathbb P}}

\newcommand{\barM}{\overline{M}}
\newcommand{\barq}{\overline{q}}

\newcommand{\mult}{\operatorname{mult}}

\newcommand{\Sym}{\operatorname{Sym}}
\def\geq{\geqslant}
\def\leq{\leqslant}

\def\ge{\geqslant}

\begin{document}
 
\title[Contact invariants for plane curves in a pencil]{Contact invariants for plane curves in a pencil}
 
\author{Ciro Ciliberto}
\address{Dipartimento di Matematica, Universit\`a di Roma Tor Vergata, Via O. Raimondo 00173 Roma, Italia}
\email{cilibert@axp.mat.uniroma2.it}

\author{Rick Miranda}
\address{Department of Mathematics, Colorado State University, Fort Collins (CO), 80523,USA}
\email{rick.miranda@colostate.edu}
 
\author{Joaquim Ro\'e}
\address{Departament de Matematiques, Universitat Aut\`onoma de Barcelona, 08193 Bellaterra (Barcelona), Catalunya}
\email{joaquim.roe@uab.cat}

\subjclass{Primary 14H50, 14N05, 14N10, 14N15; Secondary 14A25.}
 
\keywords{contact, flex, bitangent, tritangent, flex bitangent, hyperflex}
 
\begin{abstract} 
Let $\calP$ be a general pencil of curves of degree $d$ in the projective plane. 
In this paper we review the computation of the number of curves in $\calP$ that have 
a hyperflex line, a flex bitangent line and a tritangent line. 
Then we focus on the curves in the dual plane 
described by the flex tangents and the bitangents of the curves of $\calP$ 
and the curves in the original plane 
described by the flexes and the points of bitangencies of the curves in $\calP$.  
Some of these curves have been studied already: we mainly focus here on the ones that still have not been treated systematically, and we compute their degree, genus, and singularities. 

\end{abstract}

\maketitle

\tableofcontents
 
\section*{Introduction} 
Consider the projective space $|\mathcal O_{\mathbb P^2}(d)|$
(of dimension $N_d = d(d+3)/2$)
of all plane curves of degree $d$, 
i.e., all homogeneous polynomials in $[x:y:z]$ of degree $d$,
up to multiplication by a non--zero factor.
The $d=1$ case is the \emph{dual plane} $(\bbP^2)^*$ 
parametrizing lines in the plane;
it has dimension two.

Given a line $L$, a plane curve $C$, and a point $p \in C$,
we say that $L$ has \emph{contact to order at least $m$} with $C$ at $p$
if $\mult_p(L,C) \geq m$.
The $m=1$ case is just that $L$ passes through $p$;
if $p$ is a smooth point of $C$, 
the $m=2$ case is that $L$ is the tangent line to $C$ at $p$.

A \emph{contact condition} 
(of orders at least $\underline{m} = (m_1,\ldots,m_k)$)
on a plane curve is the condition that there is a line $L$
and an ordered $k$--tuple of points of $C$ 
(which may not be distinct, e.g., may be infinitely near)
such that the contact orders of $L$ with $C$ at the $i$-th point 
is at least $m_i$, for $1\leq i\leq k$.

We will assume that all $m_i \geq 2$, 
so that we are studying tangency (to possibly high order) conditions.
Form the incidence variety
\[
\calI_{\underline{m},d} =\overline{\{(L,(p_1,\ldots,p_k),C) \;|\; \mult_{p_i}(L,C) \geq m_i\}}
\subset (\bbP^2)^*\times (\bbP^2)^k\times |\mathcal O_{\mathbb P^2}(d)|
\]
where the points $p_1,\ldots,p_k$ are assumed to be distinct;
the closure operation attaches to the boundary the triples 
$(L,(p_1, \dots, p_k),C)$ where some of the points are infinitely near.
We consider the three natural projections
\begin{equation}\label{eq:incidence1}
	\begin{tikzcd}[column sep=large] 
	&\calI_{\underline{m},d}
	\arrow[dl, "\pi_1" swap] 
	\arrow[d, "\pi_2" ] 
	\arrow[dr, "\pi_3"] 
	\\ (\bbP^2)^* 
	& (\bbP^2)^k
	& {|\mathcal O_{\mathbb P^2}(d)|}\cong\bbP^{N_d}  
\end{tikzcd} 
\end{equation}
and compute the dimension of $\calI_{\underline{m},d}$
by projecting onto the first two components. 
If we fix a line $L\in (\bbP^2)^*$, the $k$ points lying on it vary in a $k$-dimensional space,
(namely $L^k$),
and the condition that a curve $C$ has contact $m_i$ at $p_i$ is $m_i$ conditions on the curve.
Hence the fiber over $(L,(p_1,\dots,p_k))$, if it is non-empty, has dimension $N_d-\sum_{i=1}^k m_i$.
Hence  
$$\dim(\calI_{\underline{m},d})= 2+k+N_d -\sum_{i=1}^k m_i=N_d +2-\sum_{i=1}^k (m_i-1).$$

The first set of cases that are of interest classically
are when, for each curve, there are a finite number of lines having the required contact.
This requires $\dim(\calI_{\underline{m},d}) = N_d$,
so that $\sum_{i=1}^k (m_i-1) = 2$.
There are only two cases: $\underline{m} = (2,2)$ or $\underline{m}=(3)$;
these are the bitangents and the flex lines.
The Pl\"ucker formulas tell us that the number of bitangents to a general curve $C$
is $d(d-2)(d^2-9)/2$ and the number of flex lines is $3d(d-2)$.
These are the degrees of the map $\calI_{\underline{m},d} \to \bbP^{N_d}$
in these two cases.

We will say more about these two cases below,
but for now let us move to the next overall case of interest,
when $\calI_{\underline{m},d}$ has dimension one less than $N_d$,
so that in the space of curves, it is one condition for the existence of a line with the required contact.
Here we must have $\sum_{i=1}^k (m_i-1) = 3$, and there are only three cases:
$\underline{m} = (4), (3,2)$, or $(2,2,2)$.

The main goal of this paper is to expose the geometry of the hypersurfaces 
$\calH_{\underline{m},d}=\pi_3(\calI_{\underline{m},d})$
in the space of curves for $\underline{m} = (4), (3,2)$, or $(2,2,2)$, 
collecting, and revisiting in a uniform way,
known results and contributing some new ones, along with a degeneration technique to deal with them, 
that we introduce and use in Section \ref {sec:recursion}.

The degrees of these hypersurfaces are given below:
\begin{equation}\label{degreetable}
\begin{array}{c|c|c}
\underline{m} & \deg(\calH_{\underline{m},d})) & \text{reference} \\
\hline
(4) & 6(d-3)(3d-2) & \text{\cite[\S 400]{Sal59}, \cite[Section 11.3.1]{3264}}, \\ 
& & \text{and Prop. \ref{prop:SC} here} \\
(3,2) & 3(d^2 + 6d - 4)(d - 3)(d - 4) & \text{Prop. 4.1 here} \\
(2,2,2) & (d^2+3d-2)(d-3)(d-4)(d-5) & \text{\cite{O}} \\
\end{array}
\end{equation}
The case of $\underline{m} =(4)$ was classically called a line of \emph{undulation} to the curve;
the modern terminology seems to be in this case that the line is a \emph{hyperflex}.
Its degree, noted above, was computed by Cayley and Salmon (\cite[\S 400]{Sal59}).
A modern proof using Chern classes is given in \cite{3264}, Section 11.3.1; in Section \ref{sec:Chern} we recall this approach and develop it further. 

In the case of $\underline{m} =(3,2)$ we will call such a line a \emph{flex bitangent}; 
again it is one condition for a curve to have one. 
In \cite[\S 400]{Sal59} 
it has been stated that the degree of $\calH_{(3,2),d}$ is $3(d-4)(3d^3+5d^2-32d+18)$. 
We found that this is wrong;
the correct number is in the table above,
 and we compute it in Section \ref {sec:fb} 
 (see also Section \ref {sse:other}). 

In the case of $\underline{m}=(2,2,2)$ 
the line is called a \emph{tritangent} line to the curve.
	The degree of $\calH_{(2,2,2),d}$ is given above, 
and we compute it in Section \ref {sec:recursion}; 
surprisingly, we could not find a classical source for this number, 
which we found only in the informal notes \cite{O} by G.~Oberdieck.
	
\smallskip

A secondary goal of this paper
is to return to the more elementary cases of $\sum_{i=1}^k(m_i-1)=2$
(the flexes and bitangents)
and consider the geometry of the curves
swept out by these special contact points and lines
as the curves move in a general pencil.
For these two cases, a general pencil $\calP$ 
of curves of degree $d$ determines, 
via the incidence variety, plane curves
\[
\calM_{\underline{m},d, \calP} 
= \pi_1(\pi_3^{-1}(\calP))\subset(\bbP^2)^* \,\, \text{and}\,\, \calB_{\underline{m},d, \calP,i}=\pi_{2,i}(\pi_3^{-1}(\calP))\subset \bbP^2,
\]
where $\pi_{2,i}: \calI_{\underline{m},d} \longrightarrow \bbP^2$ is the composition of $\pi_2$ 
with the projection of $(\bbP^2)^k$ to the $i$--th factor,  
for $i=1,\dots,k$. 
For instance, the curve $\calM_{(3),d,\calP}$ 
described by the flexes of the curves in a pencil
has been studied in \cite [\S 11.3.2] {3264} 
and in \cite {Ku}.
For the bitangent curves, a monodromy argument
shows that the two curves $\calB_{(2,2),d, \calP,1}$
and $\calB_{(2,2),d, \calP,2}$ are the same curve,
which we will denote simply by $\calB_{(2,2),d, \calP}$.

In Sections \ref {sec:unif} and \ref {sec:comp} will give numerical and geometric descriptions of some of these curves of contact lines and of contact points, 
based on a unified approach introduced in Section \ref {sec:unif}. 

Classical approaches to contact problems 
were based on more or less ad-hoc methods 
like the \emph{tact invariant} used in \cite{Sal59}. 
A general method to deal with them via Chern class computations 
is nowadays available, 
and can be found in \cite[Section 11.3.1]{3264}, 
where the calculation of the degree of $\calH_{(4),d}$ is given; 
as we said, in Section \ref{sec:Chern} we summarize this approach.
It was our goal to give a complete treatment 
of all of the various cases
for both $\sum_i(m_i-1)=2$ and $\sum_i(m_i-1)=3$,
and hence we give proofs for all the desired statements;
some of this is not original, and 
references are provided of course.


In conclusion, the original results contained in this paper are:\\
\begin{inparaenum}
\item [$\bullet$] the computation of the degree of $\calH_{(3,2),d}$ (see Section \ref {sec:fb});\\ 
\item [$\bullet$]  the computation of the degree and the arithmetic genus of the curve of bitangent points 
 $\calB_{(2,2),d,\calP}\subset \bbP^2$
 described by the tangency points of the bitangent lines to the curves in a general pencil  $\calP$ of curves of degree $d$ (see Proposition \ref {prop:bbt});\\ 
\item [$\bullet$] the computation of the degree and genus 
of $\calM_{(2,2),d, \calP}$,
 i.e., the curve of bitangent lines to curves in a general pencil $\calP$ of curves of degree $d$ (see \S \ref {ssec:bit} and \S \ref {ssec:dbf});\\
\item  [$\bullet$] the study of the singularities of $\calM_{(2,2),d, \calP}$
 (see Section \ref {sec:fb}).
\end{inparaenum}

\smallskip
We work over the complex field.

\section{Setting and Chern class computations}
\label{sec:Chern}

{
\subsection{The main diagram}
Define 
\[
\Psi = \{(L,p) \;|\; p \in L\} \subset (\bbP^2)^*\times \bbP^2
\]
and more generally, for any positive integer $k$, define
\[
\Psi_k = \overline{\{(L,(p_1,\dots,p_k)) \;|\; p_i \in L, p_i\ne p_j \text{ if } i\ne j\}} \subset (\bbP^2)^*\times (\bbP^2)^k
\]
(so $\Psi=\Psi_1$) to be the \emph{incidence varieties} of points on lines in the plane.

For every $\underline{m}$ and $d$, the incidence diagram \eqref{eq:incidence1} factors through $\Psi_k$ as 
\begin{equation*}
	\begin{tikzcd}[column sep=large] 
		&\calI_{\underline{m},d} 
		\arrow[ddr, "\pi_3"] 
		\arrow[ddl, phantom, "\Psi_k"{name=P, description}]
		\arrow[to=P, "\pi_\Psi"]\\
		& 
		\arrow[dl, from=P, "\pi'_1" swap] 
		\arrow[d, from=P, "\pi'_2" ] 
		\\
		(\bbP^2)^* 		& (\bbP^2)^k		& {|\mathcal O_{\mathbb P^2}(d)|}  
	\end{tikzcd} 
\end{equation*}
where $\pi_i=\pi_i'\circ \pi_\Psi$, for $1\leq i\leq 2$.

For every $L\in (\bbP^2)^*$ the fiber of $\pi'_1$ over $L$ is
smooth of dimension $k$, 
isomorphic to $L^k$,
so $\Psi_k$ is smooth and irreducible of dimension $\dim(\Psi_k) = 2+k$ for every positive integer $k$.

Assuming $d+1\ge \sum_{i=1}^k m_i$, 
the fiber of $\pi_\Psi$ over any point of $\Psi_k$ is isomorphic to 
a linear system of dimension $N_d-\sum_{i=1}^k m_i$, 
so $\calI_{\underline{m},d}$ is an irreducible variety of dimension 
$N_d +2-\sum_{i=1}^k (m_i-1)$, 
as stated in the Introduction. 
Therefore, the image $\calH_{\underline{m},d}:=\pi_3(\calI_{\underline{m},d})\subset|\calO_{\bbP^2}(d)|$
is also an irreducible subvariety.
Unless $\underline{m}=(2)$, the morphism
$\pi_3:\calI_{\underline{m},d}\rightarrow\calH_{\underline{m},d}$
is finite 
(note that this statement is false in positive characteristic),
so 
$$\dim (\calH_{\underline{m},d})=N_d +2-\sum_{i=1}^k (m_i-1).$$

When $\calH_{\underline{m},d}$ is a hypersurface, 
its degree counts the intersection points with a general pencil 
$\calP\subset |\calO_{\mathbb P^2}(d)|$.
This leads us to consider the following extended diagram
\begin{equation}\label{eq:incidence_pencil}
	\begin{tikzcd}[column sep=large] 
		&\calI_{\underline{m},d} 
		\arrow[ddr, "\pi_3"] 
		\arrow[ddl, phantom, "\Psi_k"{name=P, description}]
		\arrow[to=P, "\pi_\Psi"] & S 
		\arrow[l,hook,"\iota"]
		\arrow[ddr]
		\arrow[dd, phantom, "\urcorner"{name=intersect, description}, very near start] \\
		& 
		\arrow[dl, from=P, "\pi'_1" swap] 
		\arrow[d, from=P, "\pi'_2" ] 
		\\
		(\bbP^2)^* 		& (\bbP^2)^k		& {|\mathcal O_{\mathbb P^2}(d)|} &
		\mathcal{P} 		\arrow[l,hook]  
	\end{tikzcd} 
\end{equation}
where $S$ is finite if $\sum_{i=1}^k (m_i-1)=3$, 
$1$-dimensional if $\sum_{i=1}^k (m_i-1)=2$, 
and a surface for $\underline{m}=(2)$.
In all other cases it is empty.

\subsection{The bundle of principal parts}

Let us focus for the rest of this section on the $k=1$ case.
}%
The treatment here will closely follow \cite[Section 11.3.2]{3264},
but we include it for a complete treatment of the hypersurface cases,
and we later use the statements involving the flex curve.
Fix $d$ and let us define 
\[
\calL_d = (\pi'_2)^*(\calO_{\bbP^2}(d)),
\]
which is a line bundle on $\Psi$. 
For any line $L$ (considered as a point in $(\bbP^2)^*$),
the restriction of $\calL_d$ to the line $(L,-) = (\pi'_1)^*(L)$ 
is the restriction of $\calO_{\bbP^2}(d)$ to $L$.
The \emph{bundle of principal parts}
\[
\calP^\ell_{\Psi/(\bbP^2)^*}(\calL_d),
\]
for a fixed integer $\ell \geq -1$, is defined as the bundle whose stalk over a point $(L,p) \in \Psi$
is the quotient
\[
\frac{\text{germs of sections of }\calO_L(d)}
{\text{germs of sections vanishing to order at least } \ell+1 \text{ at } p}.
\]
(thus, requiring that a curve $C$ has order of contact at least $m$ with $L$ at $p$ means
that its image on the bundle of principal parts with $\ell=m-1$ vanishes).
We note that $\calP^{-1}_{\Psi/(\bbP^2)^*}(\calL_d) = 0$.

To describe this more precisely/inductively, we have short exact sequences
\[
0 \to \calL_d\otimes \Sym^\ell(\Omega_{\Psi/(\bbP^2)^*})
\to \calP^\ell_{\Psi/(\bbP^2)^*}(\calL_d) \to \calP^{\ell-1}_{\Psi/(\bbP^2)^*}(\calL_d) \to 0
\]
for every $\ell$ (see \cite [Thm. 7.2]{3264}).
This inductively implies that the total Chern polynomial for this bundle is
\begin{equation}\label{totalchernprincipalparts}
c(\calP^\ell_{\Psi/(\bbP^2)^*}(\calL_d))
= \prod_{k=0}^\ell c(\calL_d\otimes \Sym^k(\Omega_{\Psi/(\bbP^2)^*}))
\end{equation}
and so expresses this Chern polynomial in terms of $\calL_d$ and the sheaf of differentials.

We now turn to computations on the Chow ring of $\Psi$.
For this we note that $\Psi$ is the projectivization of the universal subbundle $\calS$ on $(\bbP^2)^*$ (considering this dual projective plane as a the grassmannian $\mathbb G(1,2)$).
We have
\[
A((\bbP^2)^*) = \bbC[\sigma_1]/(\sigma_1^3=0)
\]
and it is therefore a complex vector space of dimension three, with basis
$\{1,\sigma_1, \sigma_1^2 = \sigma_{11}\}$
(using Schubert cycle notation).
Then, by the description of $\Psi$ as $\bbP(\calS)$,
we have
\[
A(\Psi) = A((\bbP^2)^*)[\zeta]/(\zeta^2-\sigma_1\zeta+\sigma_{11}, \zeta^3)
\]
where $\zeta = [\mathcal L_1]$.
We note that the total Chern class of $\calS$ is $1-\sigma_1+\sigma_{11}$.

To analyze the symmetric powers of $\Omega_{\Psi/(\bbP^2)^*}$, we note that the relative tangent bundle
$\calT=\calT_{\Psi/(\bbP^2)^*}$ is a line bundle, whose first Chern class is
\[
c_1(\calT) = \sum_{i=0}^1 \binom{2-i}{1-i} c_i(\calS) \zeta^{1-i}
\]
(see \cite[Theorem 11.4]{3264}),
and therefore is equal to $2\zeta - \sigma_1$.
Hence taking symmetric powers and dualizing we see that
\[
c(\Sym^\ell(\Omega_{\Psi/(\bbP^2)^*}) = 1+\ell(\sigma_1-2\zeta)
\]
and since $c_1(\calL_d) = d\zeta$, we have
\[
c_1(\calL_d \otimes \Sym^\ell(\Omega_{\Psi/(\bbP^2)^*}) = (d-2\ell)\zeta + \ell\sigma_1.
\]

\subsection{The hyperflex invariant}\label{ssec:hyper}
For $\ell=3$, using \eqref{totalchernprincipalparts} we then have 
\[
c(\calP^3_{\Psi/(\bbP^2)^*}(\calL_d))
= \prod_{k=0}^3 (1+(d-2k)\zeta + k \sigma_1)
\]
which is
\[
[1+d\zeta][1+(d-2)\zeta+\sigma_1][1+(d-4)\zeta+2\sigma_1][1+(d-6)\zeta+3\sigma_1].
\]
The required number we are looking for is the top Chern class $c_3$.
We note that $\sigma_1^3=\zeta^3=0$, so the only terms that survive here in degree three are the two terms $\zeta^2\sigma_1$ and $\zeta\sigma_1^2$, both of which are equal to one in the top Chern group.

Multiplying this all out gives:

\begin{proposition}\label{prop:SC} One has
$$
\deg(\calH_{(4),d})= 18d^2-66d+36 = 6(d-3)(3d-2).
$$
\end{proposition}

This coincides with the Salmon-Cayley answer and the above proof has been given in \cite [\S 11.3.1]{3264}.
Note that $\deg(\calH_{(4),3})=0$ as it should be; for $d=4$ we get the well-known value of $60$.

\subsection{The curves of flex points and of flex lines}\label{ssec:flex}
For the $\ell=2$ case, the total Chern polynomial can be written as
\begin{align*}
	c &=  \prod_{k=0}^2 c(\calL_d\otimes \Sym^k(\Omega_{\Psi/(\bbP^2)^*}))
	= \prod_{k=0}^2 (1+(d-2k)\zeta + k \sigma_1) \\
	&= [1+d\zeta][1+(d-2)\zeta+\sigma_1][1+(d-4)\zeta+2\sigma_1].
\end{align*}
From this we compute
\[
c_2 = (3d^2-12d+8)\zeta^2 + (6d-8)\zeta\sigma_1 + 2\sigma_1^2.
\]
This is the class in the Chow ring of $\Psi$ that records the dependency locus of $2$ general sections.
If we take two general sections (given by two polynomials of degree $d$),
they will be dependent if and only if some linear combination is zero,
i.e., if and only if some linear combination has a flex tangent $L$ at a point $p$,
i.e., at the point $(L,p)\in\Psi$.

Consider now a general pencil $\mathcal{P}$ of curves of degree $d$. 
Let $F_{\calP,p}=B_{(3),d,\calP,1}$
be the curve in $\bbP^2$ of flex points of the curves of the pencil $\calP$,
and let 
$F_{\calP,L}=\calM_{(3),d,\calP}$
be the curve in $(\bbP^2)^*$ 
of flex lines of the curves of the pencil $\calP$.

If we intersect this class $c_2$ with $\zeta$, we are  measuring
the number of curves in the pencil $\mathcal{P}$
that have a flex point lying on a general line (whose class is $\zeta$);
this is exactly the degree of $F_{\calP,p}$.

For the dual reasons,  if we intersect the class $c_2$ with $\sigma_1$, we are  measuring
the number of curves in the pencil $\mathcal{P}$
that have a flex tangent passing through a general point of the plane, and
this is the degree of $F_{\calP,L}$.
(It is also the Pl\"ucker number of flex points on the general curve.)

Summing up we have:

\begin{proposition}\label{prop:flex} One has
\begin{gather*}
\deg(F_{\calP,p}) = 6d-6,  \deg(F_{\calP,L})= 3d(d-2), \text{and} \\
p_g(F_{\calP,L}) = p_g(F_{\calP,p}) = 12d^2-39d+25.
\end{gather*}
\end{proposition}

\begin{proof}
Indeed, we have 
\[
\deg(F_{\calP,p}) = c_2\cdot \zeta = (6d-8)\zeta^2\sigma_1 + 2 \zeta\sigma_1^2 = 6d-6
\]
since $\zeta^3=\sigma_1^3=0$ and $\zeta^2\sigma_1 = \zeta\sigma_1^2 = 1$.

Similarly
\[
\deg(F_{\calP,L}) = c_2\cdot\sigma_1 = (3d^2-12d+8)\zeta^2\sigma_1 + (6d-8)\zeta\sigma_1^2
= 3d^2-6d=3d(d-2).
\]
For the statements about the geometric genera,
see \cite[Section 11.3.3]{3264} and \cite {Ku}.
\end{proof}

The above proof can be found 
in \cite[Section 11.3.2]{3264} and in \cite {Ku}.  

\section{A Degeneration Recursion method}
\label{sec:recursion}

Let $\calP$ be a general pencil of curves of degree $d$. 
In this section we introduce a degeneration method to compute the degree of $\calM_{(2,2),d, \calP}$, 
i.e., the curve of bitangent lines to curves in $\calP$.

\subsection{The degeneration}\label{ssec:deg}  
Consider a general pencil $\calP'$ of curves of degree $d-2$,
and take a general fixed conic $\Gamma$;
we take $\bar \calP=\Gamma + \calP'$ 
as a degeneration of a general pencil $\calP$ of curves of degree $d$,
and use this to generate a recursion method to compute the degree of $\calM_{(2,2),d, \calP}$.

The bitangents to curves in $\bar\calP$ are of the following types:
\begin{enumerate}[(i)]
\item the bitangents to curves in $\calP'$;
\item the tangents to curves in $\calP'$ that are also tangent to $\Gamma$;
\item given a curve $\Gamma+C$ with $C\in \calP'$, the lines that pass through a point $p\in \Gamma\cap C$ and are also tangent to $C$ somewhere;
\item given a curve $\Gamma+C$ with $C\in \calP'$, the lines that pass through a pair of points $p,q\in \Gamma\cap C$.
\end{enumerate}

\subsection{Degeneration multiplicities}\label{ssec:dm}

In the list above, some of the bitangents have to be counted with multiplicity; namely, each curve of type (iii) (resp. type (iv)) is the limit of $\nu>1$ (resp. $\mu>1$) bitangents to curves in general pencils $\calP$. Our goal in this subsection is to compute $\nu$ and $\mu$.

\begin{proposition}\label{prop:numu} 
	One has $\nu=2$ and $\mu = 4$.
\end{proposition}

\begin{proof}
Since the computation is local, 
we may determine $\nu$ by computing the contribution 
(to the number of bitangents for a curve of degree $d$) 
of a line that contains a node of that curve,
and is tangent elsewhere.
	
We may assume that the curve is a general quartic with a node. 
This curve has  $16$ \emph{proper bitangents}, 
i.e., bitangent lines whose contact points are smooth points of the curve. 
It has also $6$ \emph{improper bitangents}, 
namely lines passing through the node and tangent elsewhere. 
The general quartic curve has $28$ bitangents. 
From this and from obvious monodromy arguments, 
it follows that the $6$ improper bitangents each count with multiplicity $2$ 
as bitangents to the curve.  

This proves that $\nu=2$, 
and since in case (iv) we have two nodes,
we conclude that $\mu = 2\nu = 4$.
\end{proof}

\begin{remark}
This degeneration allows one to compute
the Pl\"ucker number of bitangent lines.
We leave this as an exercise.
\end{remark}

\subsection{Degree of the curve of bitangent lines}\label{ssec:bit}
We now will compute the degree of the curve $\calM_{(2,2),d, \calP}$ by recursion, applying the degeneration introduced in \S \ref {ssec:deg}.

We first need the following:

\begin{lemma}\label{lem:var} 
Let $\calP$ be a general pencil of curves of degree $d$ 
and let $x$ be a general point of the plane. 
Then the curve $D_{\calP,x}$ 
described by the contact points of tangent lines 
to members of the pencil passing through $x$ has degree $2d-1$. 
\end{lemma}

\begin{proof}
The curve $D_{\calP,x}$ passes through the $d^2$
base points of the pencil with multiplicity one.
Also it intersects the general member $C$ of the pencil
in $d(d-1)$ points away from the base points 
(this is the degree of the dual curve to $C$).
Therefore 
$$
\deg (D_{\calP,x})=\frac {d(d-1)+d^2}d=2d-1.
$$
\end{proof}

\begin{proposition}\label{prop:dfd} 
For every integer $d\geq 5$, one has
\begin{align*}\label{eq:rec}
\deg (\calM_{(2,2),d, \calP}) &=\deg (\calM_{(2,2),d-2, \calP}) + 4(d-3) + 4(3d^2-19d+28) + 4(2d-5) \\
&= \deg (\calM_{(2,2),d-2, \calP})  + 4(d-2)(3d-10).
\end{align*}
\end{proposition}

\begin{proof} 
We count separately the degrees of the four curves
that comprise the four cases listed above.
This therefore involves computing, in each of these cases,
the number of bitangents that pass through a general point $x$ of the plane.

The lines in (i) and (ii) are proper bitangents, 
so they count with multiplicity $1$;
the lines of type (iii) count with multiplicity $2$
and the lines of type (iv) count with multiplicity $4$
(see Proposition \ref {prop:numu}).
 
The number of bitangents of type (i) that pass through $R$ is $\deg (C_{(2,2),d-2, \calP})$. 
The bitangents of type (ii) are nothing else than all the tangents to $\Gamma$, 
but each one has to count with multiplicity $2d-6$, 
because each such line is tangent to $2d-6$ curves in the pencil $\calP'$. 
So the contribution of the bitangents of type (ii) is $2\cdot (2d-6)=4(d-3)$. 

To compute the number of bitangents of type (iii) that pass through $R$, 
we argue as follows. 
Take a general point $p\in \Gamma$. 
There is a unique curve $C\in \calP'$ containing $p$, 
and there are $(d-2)(d-3)$ tangent lines to $C$ passing through $x$. 
These lines cut out on $\Gamma$ a total of $2(d-2)(d-3)$ points 
that form a divisor $D(p)$. 
So we have a correspondence $p\longrightarrow D(p)$ on $\Gamma$. 

The inverse of this correspondence can be obtained in the following way. 
Let $q\in \Gamma$ be a general point. 
Consider the line $L$ joining $q$ with $r$. 
There are $2d-6$ curves of $\mathcal{P}'$ tangent to $L$, 
and each of these curves cuts out $2(d-2)$ points on $\Gamma$, 
for a total of $4(d-2)(d-3)$ points;
these are exactly the points $p$ such that $q\in D(p)$. 
So the correspondence $D$ has bidegree $[2(d-2)(d-3),4(d-2)(d-3)]$.
and thus it has $6(d-2)(d-3)$ coincidence points.

The configurations that we are counting give rise to coincident points,
but there are configurations that give coincident points
that are not of type (iii);
these come from points $p \in \Gamma$
such that the unique curve $C$ in the pencil $\mathcal{P}'$ through $p$
has tangent line at $p$ passing through $x$.
To count these we consider the curve of contact points
of the tangent lines to the curves of the pencil that pass through $x$.
This curve has degree $2d-5$ by Lemma \ref {lem:var}.
Therefore it meets $\Gamma$ in $4d-10$ points, which must be removed from the
coincident count with multiplicity two;
the result is $6(d-2)(d-3) - 2(4d-10) = 6d^2-38d+56$, each counting with multiplicity two.
On the whole, the contribution of the bitangents of type (iii) is then $4(3d^2-19d+28)$.

Finally, to compute the number of bitangents of type (iv) that pass through $x$, 
we notice that the lines through $x$ cut out on $\Gamma$ a $g^1_2$ 
whereas the curves in $\mathcal{P}'$ cut out on $\Gamma$ a $g^1_{2d-4}$. 
The number of divisors of the $g^1_2$ 
that are contained in some divisor of the $g^1_{2d-4}$ is $2d-5$ 
(see \cite [p. 344]{ACGH}). 
These divisors give rise to bitangents of type (iv), 
whose contribution is therefore $4(2d-5)$. \end{proof} 

\begin{corollary}\label{cor:count} One has
$\deg(\calM_{(2,2),d, \calP})= 2d(d-2)(d-3).$
\end{corollary}

\begin{proof}
Of course $\deg(\calM_{(2,2),2, \calP}) =\deg(\calM_{(2,2),3, \calP}) = 0$;
Proposition \ref{prop:dfd} then implies that
 $\deg(\calM_{(2,2),4, \calP})=16$, which we also computed independently with SageMath.
Then using the recursion formula from  Proposition \ref {prop:dfd}, 
the assertion easily follows. 
\end{proof}

\subsection{Similar computations}\label{sse:other} 
One can use the degeneration--recursion method introduced above to make other computations similar to the one of 
$\deg(\calM_{(2,2),d, \calP})$. 

For example, using it we can compute the following recursion formula:

\begin{proposition}\label{prop:trit} For $d \geq 5$, one has
\begin{equation*}\label{eq:tritgtrecursion21}
\deg(\calH_{(2,2,2),d}) = \deg(\calH_{(2,2,2),d-2})+ 10d^4 - 112d^3 + 350d^2 - 56d - 720.
\end{equation*}
\end{proposition}

\begin{corollary}\label{cor:trit} One has
\begin{equation*}\label{eq:tritgtrecursion3}
\deg(\calH_{(2,2,2),d})=(d^2+3d-2)(d-3)(d-4)(d-5).
\end{equation*}
\end{corollary}

\begin{proof} One has  $t_3=t_4=0$. Then using Proposition \ref{prop:trit} the assertion holds. 
\end{proof}

Recall that the formula of Corollary \ref{cor:trit} 
can be found in the paper \cite{O} that appeared only online.
In a similar way one can compute $\deg(\calH_{(4),d})= 6(d-3)(3d-2)$ (computed already in \S \ref {ssec:hyper}) and $\deg(\calH_{(2,2,2),d})=(d^2+3d-2)(d-3)(d-4)(d-5)$ (see Section \ref {sec:fb} for this number). However we do not dwell here on this type of computation, since we think it has been sufficient to indicate the power of the degeneration--recursion method to compute $\deg(\calM_{(2,2),d, \calP})$.

\section{The tangency incidence correspondence for a pencil of curves}\label{sec:unif}
Let $\mathcal{P}$ be a general pencil of curves of degree $d$ in the plane:
all singular fibers have one single simple node.
The number of singular fibers is easily seen to be $\delta = 3(d-1)^2$
using an Euler number count.

\subsection{The tangency incidence correspondence}\label{ssec:esse} Consider the incidence correspondence
\[
S = \{(L,p,C) \;|\; \mult_p(L,C) \geq 2\} \subset (\bbP^2)^*\times\bbP^2\times\mathcal{P}
\]
which was suggested to us by Sheldon Katz (private communication).
The surface
$S$ fits in the diagram \eqref{eq:incidence_pencil} as $S=\pi_3^{-1}(\calP)$, giving
\begin{equation*}
	\begin{tikzcd}[column sep=large] 
		&\calI_{(2),d} 
		\arrow[ddr, "\pi_3"] 
		\arrow[ddl, phantom, "\Psi~"{name=P, description}]
		\arrow[to=P, "\pi_\Psi"] & S 
		\arrow[l,hook,"\iota"]
		\arrow[ddr]
        \arrow[dd, phantom, "\urcorner"{name=intersect, description}, very near start] \\
		& 
		\arrow[dl, from=P, "\pi'_1" swap] 
		\arrow[d, from=P, "\pi'_2" ] 
		\\
		(\bbP^2)^* 		& \bbP^2		& {|\mathcal O_{\mathbb P^2}(d)|} &
		\calP 		\arrow[l,hook]
	\end{tikzcd} 
\end{equation*}
For simplicity we write $\pi'_i=\pi'_i\circ\pi_\Psi\circ\iota$ for $1\leq i\leq	 2$ and $\pi_3=\pi_3\circ \iota$.

We first note that $\pi'_2:S\to\bbP^2$ exhibits $S$ as the blowup of the plane
at the $d^2$ base points of $\mathcal{P}$, and at the $\delta$ nodes of the singular fibers.
Hence $S$ is a smooth rational surface, with Euler number 
\[
e(S) = 3+d^2+3(d-1)^2.
\]

The first projection map $\pi'_1:S \to (\bbP^2)^*$ has, as a general fiber over a line $L$,
the collection of ordered pairs $(p,C)$ where $C$ is tangent to $L$ at $p$.
Now the pencil $\mathcal{P}$ restricts to $L$ giving a pencil $\calP_L$ which is a $g^1_d$ on $L$.
A general $g^1_d$ has exactly $2d-2$ members with a double point,
hence $\pi'_1$ is a finite map of degree $2d-2$.

\subsection{The branch curve and hyperflexes}\label{ssec:branch} The following is immediate:

\begin{lemma}\label{lem:branch}
The branch curve $B$ of $\pi'_1$ consists of\\
\begin{inparaenum}
\item[(a)] the curve $\mathcal D$ of all triples $(L,p,C)\in S$ with $L$ a flex tangent to $C$ at $p$ (note that 
$\mathcal D$ is isomorphic to 
$C_{(3),d, \calP}$), plus \\
\item[(b)] the $d^2$ lines which are the dual lines to the base points of the pencil $\mathcal{P}$.
\end{inparaenum}
\end{lemma}

The formulas developed by Enriques in \cite[p. 182]{E}, and extended by Iversen in \cite{I},
relate the invariants of a finite map of surfaces
in a way that generalizes the Hurwitz formulas for a map of curves.
We use only a simple form, referred to as \emph{Severi's formula},
written by Iversen in the Introduction of \cite{I}:
\[
e(S)-(2d-2)e((\bbP^2)^*)
= 2p_g(B) - 2 - \kappa
\]
where $\kappa$ is the number of cusps of the branch curve $B$ and $p_g$ indicates the \emph{geometric genus}, i.e., the \emph{arithmetic genus of the normalization}.
(This version is in the simplified case of having only nodes and cusps
for the branch curves, which is true in our case.)

Given Lemma \ref {lem:branch}, we see that
\begin{align*}
p_g(B) &= p_g(\mathcal D \cup \{d^2 \text{ disjoint lines}\}) \\
&= p_g(\mathcal D) - d^2= p_g(\calM_{(3),d, \calP}) - d^2\\
&= (12d^2-39d+25) -d^2 = 11d^2-39d+25.
\end{align*}

Putting things together we then see that the number of cusps of the branch curve (actually of $\mathcal D$) is
\begin{align*}
\kappa &= 2p_g(B) - 2 - e(S)+(2d-2)e((\bbP^2)^*) \\
&= (22d^2-78d+50)-2 -(3+d^2+3(d-1)^2) +(2d-2)(e) \\
&= 6(d-3)(3d-2).
\end{align*}

Since a cusp of the branch curve corresponds exactly to an undulation point on one of the curves in the pencil, this computes in a different way the number of undulation (hyperflex) points in a pencil,
and is therefore the degree of the undulation (hypersurface) locus $\calH_{(4),d}$ inside $|\mathcal O_{\bbP^2}(d)|$ (computed already in \S \ref {ssec:hyper}). .

\subsection{Double point formula and bitangents}\label{ssec:dbf} Let us now investigate the morphism
\[
f:S \to Y=(\bbP^2)^*\times\mathcal{P}, \;\;\; f(L,p,C) = (L,C)
\]
which is the joint projection map and has finite fibres.  We have that $S$ is a surface, $Y$ is a three-fold,
and the map is generically injective: given a general curve $C$ in the pencil, and a general tangent line $L$ to $C$, there is only one point of tangency $p$.
Our goal is to study the double points of this map
in order to get information about the bitangents to the curves $C$ in the pencil $\mathcal{P}$.

Let us first compute the Chern polynomial of the tangent bundle $T_Y$ on the threefold $Y$.
We have the two projections $Y\to (\bbP^2)^*$ and $Y\to \calP$,
and since $Y$ is a product, we have
\[
T_Y = p_1^*(T_{(\bbP^2)^*})\oplus p_2^*(T_\calP)
\]
where $p_1$ and $p_2$ are the two projections.
The Chow rings of the two factors are given as follows
\[
A((\bbP^2)^*) = \bbC[\barM]/\barM^3, \;\;\; A(\calP) = \bbC[\barq]/\barq^2
\]
where $\barM$ is the class of a (dual) line in $(\bbP^2)^*$,
and $\barq$ is the class of a point in the pencil $\calP \cong \bbP^1$.
We have that
\[
c(T_{(\bbP^2)^*}) = (1+\barM)^3 = 1+3\barM+3\barM^2, \;\;\;
c(T_\calP) = (1+\barq)^2 = 1+2\barq
\]
and so if we set $M = p_1^*(\barM)$ and $q= p_2^*(\barq)$, we have
\begin{equation}\label{eq:ciuno}
c(T_Y) = (1+3M + 3M^2)(1+2q)
= 1 + (3M+2q) + \text{ terms of degree } n\geq 2.
\end{equation}
We note that the class $M$ represents a divisor of the form
\[
(\{\text{ line in } (\bbP^2)^*\} \times\mathcal{P} ) \subset (\bbP^2)^*\times\mathcal{P}
\]
and the class $q$ represents a divisor of the form $(\bbP^2)^*\times \{\text{point}\}$.
The divisor classes $M$ and $q$ freely generate the Picard group of $Y$.

Next we want to pull these classes back to $S$ via the map $f:S\to Y$.
Recall that $S$ is isomorphic (via the projection to $p\in \bbP^2$)
to the blowup of the plane at the base points of the pencil,
and at the nodes of the singular members of the pencil.
Let $E_b$ be the sum of the exceptional divisors over the $d^2$ base points of the pencil, and let $E_n$ be the sum of the exceptional divisors over the $3(d-1)^2$ nodes of the singular members of the pencil.

Each of these exceptional divisors consists of triples $(L,p,C)$ where
$p$ is either a base point or the node of the singular member,
$L$ is a line through $p$,
and $C$ is then determined as the unique member of the pencil which is tangent to $L$ at $p$ (if $p$ is a base point) or which is the singular member of the pencil (if $p$ is a node).

We have
\[
E_b^2 = -d^2; \;\; E_n^2 = -3(d-1)^2;\;\; E_b \cdot E_n = 0
\]
as divisors in $S$.

Let $H$ be the pullback of the line class on $\bbP^2$ to $S$ via $\pi_2'$.

\begin{lemma}\label{lem:ident} As divisor classes on $S$, we have:
\begin{itemize}
\item[(a)] $f^*(M) = (2d-1) H - E_b - E_n$;
\item[(b)] $f^*(q) = dH - E_b$.
\end{itemize}
\end{lemma}

\begin{proof}
By symmetry, the coefficients of these divisors on all of the exceptional curves over the base points must be the same, and similarly for the exceptional curves over the nodes.
Hence in both cases we can write these divisors in terms of $H$, $E_b$, and $E_n$.

Let us address $f^*(M)$ first.
Fix a general dual line $\barM$, and a general line $H$ in the plane.
The intersection number $f^*(M)\cdot H$ records the number of triples $(L,p,C) \in S$
such that $L \in \barM$ and $p \in H$.
Now $\barM$ is the set of lines through a general point $x$ in the plane,
so this records the number of triples $(L,p,C)\in S$ such that $p \in H$ and $x \in L$.

Since $H$ is general, as the point $p$ varies in $H$, the curve $C$ in the pencil $\mathcal{P}$ through $p$ is uniquely determined by $p$; and then the line $L$ which is tangent to $C$ at $p$ is also determined.  Therefore we have a well-defined map sending the point $p\in H$ to that tangent line $L$, giving a well-defined map $g:H \to (\bbP^2)^*$.  
There are exactly $2d-2$ points $p$ on the line $H$ where the curve $C$ through $p$ is tangent to the line $H$ at $p$, i.e., there are exactly $2d-2$ points $p$ such that the map sending $p$ to the tangent line actually sends it to the line $H$.  Other than this, the map $g$ is injective: if the tangent line is not $H$, then it will meet $H$ only at the point $p$.  Therefore the image of $g$ is a rational curve in $\bbP^2$ with a unique singular point of multiplicity $2d-2$, and we conclude that it must have degree $2d-1$.
Hence this curve meets the general line $\barM$ in $2d-1$ points.
This means that there are exactly $2d-1$ tangent lines through $x$,
and we see that $f^*(M)\cdot H = 2d-1$.

Consider now an exceptional divisor $E$ lying above a base point $x$ of the pencil.
The intersection number $f^*(M)\cdot E$ records the number of triples $(L,p,C) \in S$
such that $L \in \barM$ and $L$ contains that base point $x$.
But $\barM$ is represented by the set of lines through a general point $y$;
and so the only possibility for $L$ is the line joining $x$ and $y$.
By the above description of the exceptional divisor, once $L$ is determined, then
$p$ is the base point $x$ and $C$ is the unique member of the pencil which is tangent to $L$ at $p$. Hence $f^*(M)\cdot E = 1$.

Finally if $E$ is an exceptional divisor $E$ lying over a node $x$ of a singular member $C$ of the pencil, then the points of $E$ are represented by triples $(L,x,C)$
where $L$ is any line through the node $x$.
Hence $f^*(\barM)\cdot E$ records the number of such triples where $L$ also comes from the pencil of lines $\barM$.  This is a unique line; hence again $f^*(M)\cdot E = 1$.
This proves (a).

We follow the same line of argument in the case $f^*(q)$.
We fix a general point $\barq$ in the pencil, representing a general curve $C$ in the pencil.

We first fix a general line $H$ in the plane.
The intersection number $f^*(q)\cdot H$ records the number of triples $(L,p,C)$
with $p \in C$ and $L$ the tangent line to $C$ at $p$, with $p$ lying on $H$.  This is just the degree of $C$, which is $d$.

If $E$ is an exceptional divisor lying above a base point $p$, then $f^*(q)\cdot E$
records the number of triples $(L,p,C)$ such that $L$ is the tangent line to $C$ at $p$., and there is only one such triple. 

If $E$ is an exceptional divisor lying above a node $p$, then $f^*(q)\cdot E$ must be zero;
$E$ consists of triples $(L,p,C')$ where $C'$ is the nodal member of pencil with the node at $p$, and $q$ is represented by a general member, which is not $C'$.
This proves (b). \end{proof}

The identifications of the pullbacks in Lemma \ref {lem:ident} should come as no surprise.
Indeed, (b) is rather clear: since $q$ is the pullback of a point of the pencil to $Y$,
and we further pull back to $S$, this class $f^*(q)$ should be the linear system that gives the pencil, and we exactly have the pencil here.

Similarly, (a) is the pullback of the line class for the map of $S$ to $(\bbP^2)^*$,
which as we have seen above is a finite map of degree $2d-2$.
This is consistent with the calculation of the self-intersection:
\begin{align*}
((2d-1) H - E_b - E_n)^2  &= (2d-1)^2 - d^2 - 3(d-1)^2 \\
&= 4d^2-4d+1-d^2-3d^2+6d-3 = 2d-2.
\end{align*}

For the double point formula for the morphism $f:S\to Y$, we must next compute the class $f_*[S] \in A(Y)$.
This is a divisor in $Y$, and therefore we may write this class in terms of $M$ and $q$.

\begin{lemma}\label{lem:opl}
$f_*[S] = d(d-1)M + (2d-2) q$.
\end{lemma}

\begin{proof}
In the Chow ring of $Y$, we have $M^2\cdot q=1$ and all other monomials in $M$ and $q$ are zero (since $M^3=0$ and $q^2=0$).
Let us compute $f_*[S]\cdot M^2$, which records the number of triples $(L,p,C)$
such that $L$ is a fixed line (the unique line which is the intersection of the two $M$'s).
This is the number of curves in the pencil tangent to that line $L$, which is $2d-2$,
and gives the coefficient of $q$ in the class $f_*[S]$

The value of $f_*[S]\cdot M\cdot q$ records the number of triples $(L,p,C)$
where $C$ is a fixed general member of the pencil (corresponding to $q$)
and $L$ comes from the line in $(\bbP^2)^*$ given by a general line $M$,
i.e., $L$ passes through a general point $x \in \bbP^2$.
This is the degree of the dual curve to $C$, which is $d(d-1)$,
and gives the coefficient of $M$ in the class $f_*[S]$.
\end{proof}

\begin{proposition}\label{prop:dpf} The double point locus $\bbD(f)$ for the morphism $f:S\to Y$ has class
$$
[\bbD(f)]=[2d^3-d^2-9d+6]H - [d^2+d-6]E_b - [d^2-d-2]E_n.
$$
\end{proposition}

\begin{proof} Combining  Lemmas \ref {lem:ident} and \ref {lem:opl}, we then see that
\begin{align*}
f^*f_*[S] &= d(d-1)f^*(M) + (2d-2) f^*(q) \\
&= d(d-1)[(2d-1) H - E_b - E_n] + (2d-2) [dH - E_b] \\
&= [d(d-1)(2d-1)+2d(d-1)] H - [d(d-1)+2(d-1)] E_b - d(d-1) E_n \\
&= [d(d-1)(2d+1)] H - [(d-1)(d+2)] E_b - d(d-1) E_n.
\end{align*}

By \eqref {eq:ciuno}, we also have  
\begin{align*}
f^*(c_1(T_Y)) &= f^*(3M+2q) \\
&= 3[(2d-1) H - E_b - E_n] + 2[dH - E_b] \\
&= (8d-3) H - 5 E_b - 3 E_n
\end{align*}
and since $c_1(T_X) = 3H-E_b-E_n$, the double-point formula
then gives
\begin{align*}
[\bbD(f)] &= f^*f_*[S] - f^*(c_1(T_Y)) + c_1(T_X) \\
&=  [d(d-1)(2d+1)] H - [(d-1)(d+2)] E_b - d(d-1) E_n \\
&\;\;\; - [(8d-3) H - 5 E_b - 3 E_n] + [3H-E_b-E_n] \\
&= [2d^3-d^2-d-8d+6] H -[d^2+d-2-5+1]E_b - [d^2-d-3+1]E_n \\
&= [2d^3-d^2-9d+6]H - [d^2+d-6]E_b - [d^2-d-2]E_n
\end{align*}
as claimed.  
\end{proof}

The double point locus $\bbD(f)$ contains two components:\\
\begin{inparaenum}[(i)]
\item the \emph{bitangent curve} $\mathcal B$ 
described by all triples $(L,p,C)\in S$ with $L$ bitangent of $C$ and $p$ one of the tangency points of $L$ with $C$. 
Note that the projection $(L,p,C)\mapsto L$, 
restricts to $\mathcal B$ to a double cover 
$\mathcal B\longrightarrow  \calM_{(2,2),d,\calP}$. 
Moreover $\mathcal B$ is birational to the curve $B_{(2,2),d,\calP}$ 
described by the tangency points of the bitangent lines to the curves in $\calP$;\\
\item the \emph{flex curve} $\mathcal D$ 
described by all triples $(L,p,C)\in S$ 
with $L$ flex tangent to $C$ at $p$: 
note that $\mathcal D$ can be identified with the curve $F_{\mathcal{P},p}$ of flexes of the curves in $\calP$ 
(see \S \ref {ssec:flex}).
\end{inparaenum}

By Proposition  \ref {prop:flex} and by \cite {Ku} (or with simple local computation), we have that the class of $F_{\mathcal{P},p}$  is of the form
\[
[\mathcal D]=[F_{\mathcal{P},p}]=(6d-6) H - 3E_b - 2E_n.
\]
Since $\bbD(f)$ records pairs of points with the same image by $f$, then $\mathcal D$ appears with multiplicity $2$ in $\bbD(f)$. 
Hence 
\begin{align}\label{eq:bb}
[\mathcal B] &= [\bbD(f) - 2F_{\mathcal{P},p}] \nonumber\\
&= [2d^3-d^2-21d+18]H - [d^2+d-12]E_b - [d^2-d-6]E_n \nonumber\\
&= (d-3)[(2d^2+5d-6) H - (d+4)E_b - (d+2)E_n].
\end{align}
Note that
\begin{align*}
\mathcal B\cdot f^*(M) &= (d-3)[(2d^2+5d-6) H - (d+4)E_b - (d+2)E_n] \\
&\text{ }\;\;\;\cdot [(2d-1) H - E_b - E_n] \\
&= (d-3)[(2d^2+5d-6)(2d-1) - (d+4)(d^2) - (d+2)(3(d-1)^2)] \\
&= 4d^3 - 20d^2 + 24d = 4d(d^2-5d+6) = 4d(d-2)(d-3)
\end{align*}
From this, taking into account the double cover 
$\mathcal B\longrightarrow \calM_{(2,2),d,\calP}$, 
we deduce that 
\[
\deg(\calM_{(2,2),d,\calP}) = 2d(d-2)(d-3)
\]
and we find again the result of Corollary \ref {cor:count}. 
However we can deduce more from the above considerations.

\begin{proposition}\label{prop:bbt} One has 
\begin{equation}\label{eq:degb}
\deg(\calB_{(2,2),d,\calP}) = (d-3)(2d^2+5d-6).
\end{equation}
and
$$
p_a(\calB)=3d^5 - 19d^4 + 14d^3 + 120d^2 - 240d + 73.
$$
\end{proposition}

\begin{proof} The degree of $\calB_{(2,2),d,\calP}$ is nothing but the coefficient of $H$ in $[\mathcal B]$.
The arithmetic genus of $\mathcal B$ is easily  deduced from \eqref {eq:bb}, getting
\begin{align*}
p_a(\mathcal B) &= [(d-3)(2d^2+5d-6)-1][(d-3)(2d^2+5d-6)-2]/2 \\
&\;\;\; - d^2[(d-3)(d+4)][(d-3)(d+4)-1]/2 \\
&\;\;\; -3(d-1)^2[(d-3)(d+2)][(d-3)(d+2)-1]/2 \\
&= 3d^5 - 19d^4 + 14d^3 + 120d^2 - 240d + 73
\end{align*} \end{proof}

The computations above rely on the double point formula involving the relevant Chern classes; however there is a more direct and elementary approach as well.
A priori we can write 
$$
[\mathcal B]= h H - e_b E_b - e_n E_n,
$$
and we want to compute these three coefficients.

The coefficient $e_b$ is the number of times a base point is also a point of bitangency.
Fix a base point $p$, and as the curve $C$ varies in the pencil $\mathcal{P}$,
the tangent line at $p$ also varies in the pencil of lines through $p$,
and, as we saw above, exactly three of these are flexes to $C$ at $p$.
There are $d-2$ other intersections of these lines with the members of the pencil;
therefore  if we consider the curve in $\bbP^2$ of these intersections,
it will be a curve of degree $(d-2)+3=d+1$ with a triple point at $p$.
The genus of this curve is then $d(d-1)/2 - 3$.
So by Hurwitz the number of ramification points (giving a second tangency)
is $e_b=d^2+d-12 = (d-3)(d+4)$, which agrees from \eqref {eq:bb}.

Fix a node $p$ of a singular member $C$ of the pencil.
In this case the lines through $p$ meet $C$ in $d-2$ other points,
and as they vary, we obtain a $g^1_{d-2}$ on the nodal curve $C$.
Again by Hurwitz the number of ramification points is $e_n=d^2-d-6$,
again agreeing with \eqref {eq:bb}.

Finally we need to compute the degree of $B_{(2,2),d,\calP}$.
First we note that $B_{(2,2),d,\calP}$ meets each member $C$ of the pencil
in twice the number of bitangents (each one is after all counted twice in $S$),
which is given by the classical Pl\"ucker formula as $d(d-2)(d^2-9)$.
We have seen above that $B_{(2,2),d,\calP}$ has multiplicity $e_b$ at each base point;
hence the total intersection number of the general $C$ with $B_{(2,2),d,\calP}$
is equal to the double of the number of actual bitangents, plus $d^2e_b$, giving
\[
d(d-2)(d^2-9) + d^2 (d-3)(d+4) = d(d-3)(d^2+5d-6).
\]
We get the degree of $B_{(2,2),d,\calP}$ by dividing by the degree $d$ of $C$,
and hence we find again \eqref {eq:degb}. 

\section{Flex bitangents}\label{sec:fb}
Here we compute:

\begin{proposition}\label{prop:fb} One has
$
\deg (\calH_{(3,2),d})=3(d^2 + 6d - 4)(d - 3)(d - 4).
$
\end{proposition}

\begin{proof}
Let $\calP$ be a general pencil of curves of degree $d$ and consider again the surface $S$ introduced in \S \ref {ssec:dbf}. Let us consider the intersection of the two curves $\mathcal B$ (the bitangent curve) and $\mathcal D$ (the flex curve) on $S$ introduced on p. \pageref {prop:bbt}. 
  
Among these intersections are the undulation (hyperflex) points
each counted with multiplicity two, because they are cusps of $\mathcal D$ (see \S \ref {ssec:branch}).
Hence the number of actual flex bitangents is be $\calB\cdot \mathcal D - 2\deg(\calH_{(4),d})$
where $\deg(\calH_{(4),d}) = 6(d-3)(3d-2)$ is the previously computed number of undulation points (see Proposition \ref {prop:SC}).

We make the computation on the surface $S$:
\begin{align*}
\mathcal B\cdot \mathcal D - 2\deg(\calH_{(4),d})&= 
(d-3)[(2d^2+5d-6) H - (d+4)E_b - (d+2)E_n] \\
&\;\;\; \cdot [(6d-6) H - 3E_b - 2E_n ] -2\deg(\calH_{(4),d})\\
&= (d-3)(2d^2+5d-6)(6d-6) - 3(d-3)(d+4)d^2 \\
&\;\;\; - 6(d-3)(d+2)(d-1)^2 - 12(d-3)(3d-2)\\
&=3d^4 - 3d^3 - 102d^2 + 300d - 144 \\
&= 3(d^2 + 6d - 4)(d - 3)(d - 4).
\end{align*}

\end{proof}

As noted in the Introduction, this result differs from the Cayley-Salmon formula claimed in \cite{Sal59}[\S 400], although no proper proof is given there. The formulas differ for $d=5$ already, and we have checked that our formula is correct (using SageMath) in this case, and also for $d=6$.  Separately, using a recursion method similar to the one used here in Section \ref{sec:recursion}, we also verified our formula for $d=5,6,7$, and therefore, despite the usual reliability of Cayley and Salmon in these matters, we are confident in the above result.

\section{Projections of the bitangent curve}\label{sec:comp}

Let, as usual, $\calP$ be a general pencil of curves of degree $d$. Recall the surface $S$ introduced in \S \ref {ssec:esse},
and the bitangent curve $\calB$ on $S$ (see  \S \ref {ssec:dbf}). The projection of $\calB$ via the first projection 
$p_1: S\longrightarrow (\bbP^2)^*$ is the curve
$\calM_{(2,2),d,\calP}$
 described by all bitangent lines to curves in the pencil $\calP$. The projection of $\calB$ via the second projection 
 $p_2: S\longrightarrow \bbP^2$ is the curve $\calB_{(2,2),d,\calP}$ described by the tangency points of the bitangent lines to the curves in $\calP$. 

We also have the third projection $\pi_3:S \to \calP$ mapping $\calB$ to the $\bbP^1$ which is the pencil $\mathcal{P}$ of the curves.
Let us analyze each of these projections in this section.

\subsection{The third projection}\label{ssec:gen}
Let $\tilde {\mathcal B}$ be the normalization of $\mathcal B$, which is a smooth curve whose (arithmetic) genus
is the geometric genus of $\mathcal B$. 

\begin{proposition}\label{prop:pg} One has
$$
p_g(\calB)=p_a(\tilde {\calB})
=8d^4 -13d^3 - 195d^2 + 582d - 287.
$$
\end{proposition}

\begin{proof}
The projection to $\calP$ induces a morphism 
$\pi:\tilde {\calB} \to \bbP^1$.
Each member $C$ of the pencil has $d(d-2)(d^2-9)/2$ bitangents,
and so since each bitangent line $L$ occurs twice in $\calB$
(with the second coordinate being the two tangent points),
we see that the degree of the map is $d(d-2)(d^2-9)$.
This gives a linear series of dimension one and degree $d(d-2)(d^2-9)$ on $\tilde {\calB}$;
let us compute the number of branch points for the map.

First, each line of undulation (where $\mult_p(L,C) \geq 4$) is a branch point and the number of these is
\[
6(d-3)(3d-2)
\]
(see Proposition \ref {prop:SC}).

Second, let us consider branching over a nodal member $C$ of $\calP$.
Each line $L$ passing through the node $p$, 
and tangent to the curve at another point $q$,
counts as a bitangent point $(L,p,C)$ and $(L,q,C)$, 
both with multiplicity two,
and therefore both give branch points for the map 
(see \S \ref {ssec:dm} 
and specifically Proposition \ref {prop:numu}).
The number of these tangent lines at the node $p$ is given by Hurwitz,
since these lines cut out a $g^1_{d-2}$ on the nodal curve,
which has genus $(d-1)(d-2)/2 - 1$:
this number is $d^2-d-6$.
So the total number of these branch points is therefore
\[
2(3(d-1)^2) (d^2-d-6)= 6d^4 - 18d^3 - 18d^2 + 66d - 36.
\]

The third and final contribution to branching for the map is given by the flex bitangents,
whose number is
\[
3d^4 - 3d^3 - 102d^2 + 300d - 144
\]
(see Proposition \ref {prop:fb}).  
We claim that each of them  contributes $4$ to the ramification of $\pi$.
If $p$ is the flex point and $q$ is the tangent point for a line $L$ to a member $C$,
then both $(L,p,C)$ and $(L,q,C)$ are double branch points of the map to $\calP$,
each contributing two to the ramification count; hence the total contribution is
\[
4(3d^4 - 3d^3 - 102d^2 + 300d - 144).
\]
Thus the total amount of ramification is
\begin{align*}
R(\pi) &= [6(d-3)(3d-2)]+[6d^4 - 18d^3 - 18d^2 + 66d - 36] \\
&\;\;\; +[4(3d^4 - 3d^3 - 102d^2 + 300d - 144)] \\
&= 9d^4 - 21d^3 - 102d^2 + 300d - 144
\end{align*}
From this, by Hurwitz formula, we compute the genus of $\tilde {\calB}$, i.e., the geometric genus of $B$:
\begin{align*}
p_g(B) &= [-2(\deg(\pi_3))+R(\pi_3)+2]/2 \\
&= [(-2)(d)(d-2)(d^2-9) + (9d^4 - 21d^3 - 102d^2 + 300d - 144)+2]/2 \\
&= 8d^4 -13d^3 - 195d^2 + 582d - 287.
\end{align*}
\end{proof}

\subsection{The singularities of the bitangent curve}\label{ssec:sing}
The singularities of $\calB$ come from tritangent lines to members $C$ of the pencil $\mathcal{P}$.
Indeed, each tritangent line $L$ gives three members of $\calB$,
and each is a double point of $\calB$. 
The number of tritangent lines to curves in $\calP$ is $\deg(\calH_{(2,2,2),d})$. So we have
\begin{align*}
3\deg(\calH_{(2,2,2),d}) &= p_a(\mathcal B) - p_g(\mathcal B) \\
&= [3d^5 - 19d^4 + 14d^3 + 120d^2 - 240d + 73] \\
&\;\;\; - [8d^4 -13d^3 - 195d^2 + 582d - 287] \\
&= 3d^5 - 27d^4 +27d^3 +315d^2 -822d +360
\end{align*}
hence 
\[
\deg(\calH_{(2,2,2),d})= (d^2+3d-2)(d-3)(d-4)(d-5),
\]
and we find again the result in Corollary \ref {cor:count}.

\subsection{The second projection} 
The  projection of $\calB$ via the second projection 
$p_2: S\longrightarrow \bbP^2$ is the curve $\calB_{(2,2),d,\calP}$
described by the tangency points of the bitangent lines to the curves in $\calP$. 
Next, using this projection,  we can describe the singularities of $\calB_{(2,2),d,\calP}$.

Given an element $(L,p,C) \in \calB$,
and a general point $p$,
there is a unique curve in the pencil $\calP$ through $p$,
and a unique tangent line $L$ there,
and so the map $\pi_2:\calB\to \calB_{(2,2),d,\calP}$ is birational.
The map fails to be injective
only in two cases.

First, the pre-image of each of the $d^2$ base points $p$ 
consists of $(d-3)(d+4)$ points of $\calB$:
this computation we made above, as the number of elements of the pencil
whose tangent line at a base point is tangent at a second point.
(This is the coefficient of $E_b$ in the class of $\calB$ on $S$).
This gives a singularity of $B_{(2,2),d,\calP}$ which is resolved on $\calB$.

Second, the pre-image of each of the $3(d-1)^2$ nodes
of the singular members of the pencil $\mathcal{P}$
consists of $(d-3)(d+2)$ distinct points of $B$:
these correspond to the lines through the node
which are tangent elsewhere to the curve.
This gives a singularity of $B_{(2,2),d,\calP}$ which is resolved on $\mathcal B$.

We note that if $L$ is a tritangent line to a curve $C$
then on $\calB$ we have three elements $(L,p,C)$,
and at each of these $\calB$ has an ordinary node.
This node is reproduced on $B_{(2,2),d,\calP}$,
and so there are three nodes on $B_{(2,2),d,\calP}$ for each tritangent line
(occuring at the three points of tangency).
These are the only additional singularities of $B_{(2,2),d,\calP}$.

Since the map is birational, the two curves have the same geometric genus, namely
\[
p_g(\calB) = p_g(B_{(2,2),d,\calP}) = 8d^4 -13d^3 - 195d^2 + 582d - 287
\]
(see Proposition \ref {prop:pg}).

\subsection{The first projection}
As above consider $\calM_{(2,2),d,\calP}$
the image of $\calB$ under the first projection to the dual plane. 
We have seen (see Corollary \ref {cor:count}) that
\[
\deg(\calM_{(2,2),d,\calP}) = 2d(d-2)(d-3)
\]
and the map from the normalization $\tilde {\calB}$ to the normalization ${\tilde \calM_{(2,2),d,\calP}}$ is a double cover.
This map is branched only at the lines of undulation.
By Hurwitz applied to the double cover 
$\tilde {\calB}\longrightarrow {\tilde \calM_{(2,2),d,\calP}}$, 
we have  that
\[
2p_g(\calB)-2 = 2(2p_g(\calM_{(2,2),d,\calP})-2 + \deg (\calH_{(4),d})
\]
and we know $p_g(\calB)$ from above.  
We conclude:
\begin{proposition}\label{prop:gen} One has 
$$
p_g(\calM_{(2,2),d,\calP})=4d^4 - (13/2)d^3 - 102d^2 + (615/2)d - 152.
$$
\end{proposition}

\begin{proof} One has
\begin{align*}
p_g(\calM_{(2,2),d,\calP}) 
&= [2p_g(\calB)-2 - \deg (\calH_{(4),d}) + 4]/4 \\
&= [2(8d^4 -13d^3 - 195d^2 + 582d - 287) - (6(d-3)(3d-2)) +2 ]/4 \\
&= 4d^4 - (13/2)d^3 - 102d^2 + (615/2)d - 152
\end{align*}
as claimed. 
\end{proof}

\begin{remark}\label{rem:fin} 
We now remark that $\calM_{(2,2),d,\calP}$ is a plane curve with $\deg(\calH_{(2,2,2),d})$ triple points,
corresponding to the tritangent lines to curves in the pencil $\calP$.
Hence if these are the only singularities of $\calM_{(2,2),d,\calP}$,
the geometric genus would be the arithmetic genus minus $3\deg(\calH_{(2,2,2),d})$.
However (computing arithmetic genus from the degree) we see that
\begin{align*}
&p_a(\calM_{(2,2),d,\calP}) - 3\deg(\calH_{(2,2,2),d}) - p_g(\calM_{(2,2),d,\calP})\\
 &= 
(\deg(\calM_{(2,2),d,\calP}))-1)(\deg(\calM_{(2,2),d,\calP})-2)/2 \\
&\text{ }\;\; - 3\deg(\calH_{(2,2,2),d}) - p_g(\calM_{(2,2),d,\calP}) \\
&= 2d^6 - 23d^5 + 97d^4 - (287/2)d^3 - 126d^2 + (993/2)d - 207
\end{align*}
which is not zero.
This shows that $\calM_{(2,2),d,\calP}$ has additional singularities besides the triple points
coming from the tritangents.

We conjecture that these singularities arise from lines that are bitangent to more than one curve in the pencil, and should give ordinary nodes to $\calM_{(2,2),d,\calP}$, 
the number of which
should be the formula above.
For example, for $d=3$, this number is $12$:
are there $12$ common flex lines to members of a general pencil of cubics?
\end{remark}

\end{document}